\documentclass[12pt]{article}
\usepackage{amssymb,amsmath, amsthm,latexsym, verbatim, amscd}
\usepackage{here}

  \newtheorem{theorem}{Theorem}
  \newtheorem{proposition}[theorem]{Proposition} %
  \newtheorem{lemma}[theorem]{Lemma} %

  \newtheorem{definition-theorem}[theorem]{Definition-Theorem}
\theoremstyle{definition} %
  \newtheorem{definition}[theorem]{Definition} %
  \newtheorem{example}[theorem]{Example} %
  \newtheorem{problem}[theorem]{Problem}

  \newtheorem{fact}[theorem]{Fact}
  
  \newtheorem{conjecture}[theorem]{Conjecture}
  \newtheorem{remark}[theorem]{Remark} %

\newcommand{\invHom}[3]{\operatorname{Hom}_{#1}(#2,#3)}

\begin{document}
\title{Multiplicity in restricting minimal representations}
\author{Toshiyuki KOBAYASHI
\footnote{
Graduate School of Mathematical Sciences, 
The University of Tokyo, 
3-8-1 Komaba, Tokyo 153-8914, Japan.  
}}

\maketitle

\setcounter{section}{0}

\begin{abstract}
We discuss the action of a subgroup on small nilpotent orbits, 
 and prove a bounded multiplicity property
 for the restriction of minimal representations
 of real reductive Lie groups
 with respect to arbitrary reductive symmetric pairs.  
\end{abstract}



\section{Statement of main results}
\label{sec:Intro}

This article is a continuation
 of our work \cite{K95, K14, K21, K22, K22PJA, KOP, xktoshima}
 that concerns the restriction 
 of irreducible representations $\Pi$ of reductive Lie groups $G$
 to reductive subgroups $G'$
 with focus on the {\it{bounded multiplicity property}}
 of the restriction $\Pi|_{G'}$
 (Definition \ref{def:bdd}).  
In this article we highlight the following specific setting:

\begin{enumerate}
\item[$\cdot$]
$(G, G')$ is an arbitrary reductive symmetric pair;
\item[$\cdot$]
$\Pi$ is of the smallest Gelfand--Kirillov dimension.  
\end{enumerate}

We refer to \cite{xKVogan2015}
 for some motivation and perspectives
 in the general branching problems, 
 see also Section \ref{sec:2}
 for some aspects regarding finite/bounded multiplicity properties
 of the restriction.

To be rigorous about \lq\lq{multiplicities}\rq\rq\
 for infinite-dimensional representations, 
 we need to fix the topology
 of the representation spaces.  
For this, 
 let $G$ be a real reductive Lie group,
 ${\mathcal{M}}(G)$ the category
 of smooth admissible representations
 of $G$
 of finite length with moderate growth, 
 which are defined on Fr{\'e}chet topological vector spaces
 \cite[Chap.~11]{WaI}.  
We denote by $\operatorname{Irr}(G)$
 the set of irreducible objects
 in ${\mathcal{M}}(G)$.

Suppose that $G'$ is a reductive subgroup in $G$.  
For $\Pi \in {\mathcal{M}}(G)$, 
 the {\it{multiplicity}}
 of $\pi \in \operatorname{Irr}(G')$
 in the restriction $\Pi|_{G'}$ is defined by 
\[
 [\Pi|_{G'}:\pi]:=\dim_{\mathbb{C}} \invHom{G'}{\Pi|_{G'}}{\pi}
 \in {\mathbb{N}} \cup \{\infty\}, 
\]
where 
$\invHom{G'}{\Pi|_{G'}}{\pi}$ denotes the space 
 of {\it{symmetry breaking operators}}, 
 {\it{i.e.}},  
 continuous $G'$-homomorphisms
 between the Fr{\'e}chet representations.  
For non-compact $G'$, 
 the multiplicity $[\Pi|_{G'}:\pi]$ may be infinite
 even when $G'$ is a maximal subgroup of $G$, 
 see Example \ref{ex:fmult} below.

By a {\it{reductive symmetric pair}} $(G,G')$, 
 we mean 
 that $G$ is a real reductive Lie group
 and that $G'$ is an open subgroup in the fixed point group $G^{\sigma}$
 of an involutive automorphism $\sigma$
 of $G$.  
The pairs $(SL(n,{\mathbb{R}}), SO(p,q))$
 with $p+q=n$, 
 $(O(p,q), O(p_1, q_1) \times O(p_2, q_2))$
 with $p_1 + p_2=p$, $q_1 + q_2=q$, 
 and the {\it{group manifold case}}
 $({}^{\backprime}G \times {}^{\backprime}G, \operatorname{diag}({}^{\backprime}G))$
 are examples.  
For a reductive symmetric pair $(G,G')$, 
 the subgroup $G'$ is maximal amongst reductive subgroups
 of $G$.

One may ask for which pair $(G,G')$
 the finite multiplicity property
\begin{equation}
\label{eqn:PP}
 [\Pi|_{G'}:\pi]<\infty, 
\quad
 {}^{\forall} \Pi \in \operatorname{Irr}(G), 
  {}^{\forall} \pi \in \operatorname{Irr}(G')
\end{equation}
 holds.  
Here are examples
 when $(G,G')$ is a reductive symmetric pair:

\begin{example}
[{\cite{K95, xKMt}}]
\label{ex:fmult}
(1)\enspace
For the symmetric pair $(SL(n,{\mathbb{R}}), SO(p,q))$
 $(p+q=n)$, 
 the finite multiplicity property \eqref{eqn:PP} holds
 if and only if one of the following conditions holds:
 $p=0$, $q=0$, or $p=q=1$.  
\par\noindent
(2)\enspace
For the pair 
 $(O(p,q), O(p_1, q_1) \times O(p_2, q_2))$
 ($p_1 + p_2=p$, $q_1 + q_2=q$), 
 the finite multiplicity property \eqref{eqn:PP} holds
 if and only if one of the following conditions holds:
 $p_1+q_1 =1$, 
 $p_2+q_2 =1$, $p=1$, or $q=1$.  
\par\noindent
(3)\enspace
For the group manifold case 
 $({}^{\backprime} G \times {}^{\backprime} G, \operatorname{diag}({}^{\backprime} G))$ where ${}^{\backprime} G$ is a simple Lie group, 
 the finite multiplicity property \eqref{eqn:PP} holds
 if and only if ${}^{\backprime} G$ is compact
 or is locally isomorphic to $SO(n,1)$.  
\end{example}

See Fact \ref{fact:KO} (2) for a geometric criterion
 of the pair $(G,G')$
 to have the finite multiplicity property \eqref{eqn:PP}.  
A complete classification of such symmetric pairs
 $(G,G')$ was accomplished
 in Kobayashi--Matsuki \cite{xKMt}.

On the other hand, 
 if we confine ourselves 
 only to \lq\lq{small}\rq\rq\ representations $\Pi$
 of $G$, 
 there will be a more chance 
 that the multiplicity  $[\Pi|_{G'}:\pi]$ becomes finite, 
 or even stronger, 
 the restriction $\Pi|_{G'}$ has the bounded multiplicity property
 in the following sense:

\begin{definition}
\label{def:bdd}
Let $\Pi \in {\mathcal{M}}(G)$.  
We say the restriction $\Pi|_{G'}$ has 
 the {\it{bounded multiplicity property}}
 if $m(\Pi|_{G'})<\infty$, 
 where we set
\begin{equation}
\label{eqn:msup}
m(\Pi|_{G'}):=\underset{\pi \in \operatorname{Irr}(G')}\sup
 \dim_{\mathbb{C}} 
 \invHom{G'}{\Pi|_{G'}}{\pi}
 \in {\mathbb{N}} \cup \{\infty\}.  
\end{equation}

\end{definition}

In the series of the papers, 
 we have explored the bounded multiplicity property
 of the restriction $\Pi|_{G'}$
 not only uniformly with respect to $\pi \in \operatorname{Irr}(G')$
 for the subgroup $G'$ 
 but also uniformly with respect to $\Pi \in {\mathcal{M}}(G)$, 
 {\it{e.g.}}, 
 either $\Pi$ runs over the whole set $\operatorname{Irr}(G)$
 \cite{K95, K14, xktoshima}
 or $\Pi$ belongs to certain family
 of \lq\lq{relatively small}\rq\rq\ representations of the group $G$
 \cite{K21, K22, K22PJA, KOP}.  
See Section \ref{sec:2}
 for some general results, 
 which tell that the smaller $\Pi$ is,
 the more subgroups $G'$ tends to satisfy 
 the bounded multiplicity property $\Pi|_{G'}$.  
In this article, 
 we highlight the extremal case
 where $\Pi$ is the \lq\lq{smallest}\rq\rq, 
 and give the bounded multiplicity theorems
 for {\it{all}} symmetric pairs $(G,G')$.

What are \lq\lq{small  representations}\rq\rq\
 amongst infinite-dimensional representations?   
For this, 
 the Gelfand--Kirillov dimension serves
 as a coarse measure of the \lq\lq{size}\rq\rq\
 of representations.  
We recall that for $\Pi \in {\mathcal{M}}(G)$
 the Gelfand--Kirillov dimension $\operatorname{DIM}(\Pi)$
 is defined as half the dimension
 of the associated variety of ${\mathcal{I}}$
 where ${\mathcal{I}}$ is the annihilator of $\Pi$
 in the universal enveloping algebra $U({\mathfrak{g}}_{\mathbb{C}})$
 of the complexified Lie algebra ${\mathfrak{g}}_{\mathbb{C}}$.  
The associated variety of ${\mathcal{I}}$ is 
 a finite union 
 of nilpotent coadjoint orbits
 in ${\mathfrak{g}}_{\mathbb{C}}^{\ast}$.

We recall for a complex simple Lie algebra ${\mathfrak{g}}_{\mathbb{C}}$, 
 there exists a unique non-zero minimal nilpotent
 $(\operatorname{Int}{\mathfrak{g}}_{\mathbb{C}})$-orbit
 in ${\mathfrak{g}}_{\mathbb{C}}^{\ast}$, 
 which we denote by ${\mathbb{O}}_{\operatorname{min}, {\mathbb{C}}}$.  
The dimension of ${\mathbb{O}}_{\operatorname{min}, {\mathbb{C}}}$
 is known as below, 
 see \cite{C93} for example.  
We set
 $n({\mathfrak{g}}_{\mathbb{C}})$
 to be half the dimension of ${\mathbb{O}}_{\operatorname{min}, {\mathbb{C}}}$.

\begin{figure}[H]
\begin{center}
\begin{tabular}{c|ccccccccc}
${\mathfrak{g}}_{\mathbb{C}}$
& $A_n$
& $B_n$ $(n \ge 2)$
& $C_n$
& $D_n$
& ${\mathfrak{g}}_{2}^{\mathbb{C}}$
& ${\mathfrak{f}}_{4}^{\mathbb{C}}$
& ${\mathfrak{e}}_{6}^{\mathbb{C}}$
& ${\mathfrak{e}}_{7}^{\mathbb{C}}$
& ${\mathfrak{e}}_{8}^{\mathbb{C}}$
\\
\hline
$n({\mathfrak{g}}_{\mathbb{C}})$
& $n$
& $2n-2$
& $n$
& $2n-3$
& $3$
& $8$
& $11$
& $17$
& $29$
\\
\end{tabular}
\end{center}
\end{figure}

For the rest of this section, 
 let $G$ be a non-compact connected simple Lie group
 without complex structure.  
This means
 that the complexified Lie algebra ${\mathfrak{g}}_{\mathbb{C}}$
 is still a simple Lie algebra.  
By the definition, 
 the Gelfand--Kirillov dimension has the following property:
$\operatorname{DIM}(\Pi)=0$
 $\iff$ $\Pi$ is finite-dimensional, 
  and 
\begin{equation}
\label{eqn:GKbdd}
  n({\mathfrak{g}}_{\mathbb{C}}) \le \operatorname{DIM}(\Pi)
  \le \frac 1 2 (\dim{\mathfrak{g}}-\operatorname{rank}{\mathfrak{g}}), 
\end{equation}
 for any infinite-dimensional $\Pi \in \operatorname{Irr}(G)$.
In this sense, 
 if $\Pi \in \operatorname{Irr}(G)$ satisfies
 $\operatorname{DIM}(\Pi)=n({\mathfrak{g}}_{\mathbb{C}})$, 
 then such $\Pi$ is thought
 of as the \lq\lq{smallest}\rq\rq\
 amongst infinite-dimensional irreducible representations of $G$.  


In this article, 
 we prove the following bounded multiplicity theorem
 of the restriction:

\begin{theorem}
\label{thm:Joseph}
If the Gelfand--Kirillov dimension
 of $\Pi \in \operatorname{Irr}(G)$ is $n({\mathfrak{g}}_{\mathbb{C}})$, 
 then $m(\Pi|_{G'})<\infty$ 
 for any symmetric pair $(G,G')$.    
\end{theorem}

For $\Pi_1, \Pi_2 \in \operatorname{Irr}(G)$, 
 we consider the tensor product representation
 $\Pi_1 \otimes \Pi_2$, 
 and define the upper bound of the multiplicity
 in $\Pi_1 \otimes \Pi_2$ by 
\[
m(\Pi_1 \otimes \Pi_2):=
\underset{\Pi \in \operatorname{Irr}(G)} \sup \dim_{\mathbb{C}}
\invHom{G}
{\Pi_1 \otimes \Pi_2}
{\Pi}
 \in {\mathbb{N}} \cup \{\infty\}.
\]

The tensor product representation 
 of two representations is a special case
 of the restriction
 with respect to symmetric pairs.  
We also prove
 the bounded multiplicity property
 of the tensor product:

\begin{theorem}
\label{thm:tensormin}
If the Gelfand--Kirillov dimensions of 
 $\Pi_1, \Pi_2 \in \operatorname{Irr}(G)$
 are $n({\mathfrak{g}}_{\mathbb{C}})$, 
 then one has 
$m(\Pi_1 \otimes \Pi_2)<\infty$.  
\end{theorem}

\begin{remark}
Since the upper bound 
 of the multiplicity $m(\Pi|_{G'})$ is defined
 in the category 
 of admissible representations
 of moderate growth, 
 $m(\Pi|_{G'})$ also gives an upper bound
 in the category of unitary representations
 where the multiplicity in the direct integral
 of irreducible unitary representations is defined
 as a measurable function
 on the unitary dual of the subgroup $G'$.  
\end{remark}

These results apply to \lq\lq{minimal representations}\rq\rq\
 of $G$, 
 which we recall now.  
For a complex simple Lie algebra
 ${\mathfrak{g}}_{\mathbb{C}}$
 other than ${\mathfrak{s l}}(n,{\mathbb{C}})$, 
 Joseph \cite{J79}
 constructed a completely prime two-sided primitive ideal ${\mathcal{J}}$
 in $U({\mathfrak{g}}_{\mathbb{C}})$, 
whose associated variety
 is 
 the closure of the minimal nilpotent orbit
 ${\mathbb{O}}_{\operatorname{min}, {\mathbb{C}}}$.  
See also \cite{GS04}.

\begin{definition}
[minimal representation, 
 see {\cite{GS05}}]
\label{def:minrep}
An irreducible admissible representation $\Pi$ of $G$
 is called a {\it{minimal representation}}
 if the annihilator
 of the $U({\mathfrak{g}}_{\mathbb{C}})$-module $\Pi$
 is the Joseph ideal ${\mathcal{J}}$ of $U({\mathfrak{g}}_{\mathbb{C}})$.  
\end{definition}

The two irreducible components
 of the Segal--Shale--Weil representation are classical examples
 of a minimal representation
 of the metaplectic group $M p(n,{\mathbb{R}})$, 
 the connected double cover
 of the real symplectic group $S p(n,{\mathbb{R}})$, 
 which play a prominent role 
 in number theory.  
The solution space of the Yamabe Laplacian on ${\operatorname{S}}^p \times {\operatorname{S}}^q$ gives 
 the minimal representation of the conformal transformation group $O(p+1,q+1)$
 when $p+q$ $(\ge 6)$ is even (\cite{KOr}).  
In general, 
 there are at most four minimal representations
 for each connected simple Lie group $G$ if exist, 
 and they were classified \cite{GS05, Ta}.

By the definition of the Joseph ideal, 
 one has  $\operatorname{DIM}(\Pi)=n({\mathfrak{g}}_{\mathbb{C}})$
 if $\Pi$ is a minimal representation.  
Thus Theorems \ref{thm:Joseph} and \ref{thm:tensormin} imply the following:

\begin{theorem}
\label{thm:minbdd}
Let $\Pi$ be a minimal representation of $G$.  
Then the restriction $\Pi|_{G'}$ has the bounded multiplicity property 
 $m(\Pi|_{G'})<\infty$
 for any symmetric pair $(G,G')$.  
\end{theorem}

\begin{theorem}
\label{thm:tensor}
Let $\Pi_1$, $\Pi_2$ be minimal representations of $G$.  
Then the tensor product representation has the bounded multiplicity property
 $m(\Pi_1 \otimes \Pi_2)<\infty$.  
\end{theorem}

\begin{example}
The tensor product representation of the two copies
 of the Segal--Shale--Weil representations
 of the metaplectic group $M p(n,{\mathbb{R}})$
 is unitarily equivalent to the 
 phase space representation
 of $S p(n, {\mathbb{R}})$
 on $L^2({\mathbb{R}}^{2n})$
 via the Wigner transform, 
 see \cite[Sect.\ 2]{KOPU09} for instance.  
\end{example}

In general, 
 it is rare
 that the restriction $\Pi|_{G'}$ 
 of $\Pi \in {\mathcal{M}}(G)$
 is {\it{almost irreducible}}
 in the sense
 that the $G'$-module $\Pi|_{G'}$ remains irreducible
 or a direct sum of finitely many 
 irreducible representations of $G'$.  
In \cite[Sect.\ 5]{K11Zuckerman}, 
 we discussed such rare phenomena
 and gave a list
 of the triples $(G,G',\Pi)$
 where the restriction $\Pi|_{G'}$ is almost irreducible, 
 in particular, 
 in the following settings:
 $\Pi\in {\mathcal{M}}(G)$ is a degenerate principal series representation 
 or Zuckerman's derived functor module 
 $A_{\mathfrak{q}}(\lambda)$, 
 which is supposed to be
 a \lq\lq{geometric quantization}\rq\rq\
 of a hyperbolic coadjoint orbit
 or an elliptic coadjoint orbit, 
 respectively, 
 in the orbit philosophy, 
 see \cite[Thms.\ 3.8 and 3.5]{K11Zuckerman}.  
As a corollary of Theorem \ref{thm:minbdd}, 
 we also prove the following theorem 
 where $\Pi$ is \lq\lq{attached to}\rq\rq\
 the minimal nilpotent coadjoint orbit
 ${\mathbb{O}}_{\operatorname{min}, {\mathbb{C}}}$. 

\begin{theorem}
\label{thm:irr}
Suppose that $(G,G')$ is a symmetric pair
 such that the complexified Lie algebras 
 $({\mathfrak{g}}_{\mathbb{C}}, {\mathfrak{g}}_{\mathbb{C}}')$
 is in the list 
 of Proposition \ref{prop:OKO} (vi).  
Then the restriction $\Pi|_{G'}$ is almost irreducible
 if $\Pi$ is a minimal representation of $G$.  
\end{theorem}

\begin{example}
For the following symmetric pairs 
 $({\mathfrak{g}}, {\mathfrak{g}}')$, 
 there exists a minimal representation $\Pi$
 of some Lie group $G$
 with Lie algebra ${\mathfrak{g}}$
 ({\it{e.g.,}} $G=Mp(n,{\mathbb{R}})$
 for ${\mathfrak{g}}= {\mathfrak{s p}}(n,{\mathbb{R}})$), 
 and Theorem \ref{thm:irr}
 applies to $(G,G',\Pi)$.  

\par\noindent
$\bullet$\enspace
$({\mathfrak{s p}}(p+q,{\mathbb{R}}), {\mathfrak{s p}}(p,{\mathbb{R}})\oplus{\mathfrak{s p}}(q,{\mathbb{R}}))$

\par\noindent
$\bullet$\enspace
$({\mathfrak{s o}}(p,q), {\mathfrak{s o}}(p-1,q))$
 or 
$({\mathfrak{s o}}(p,q),{\mathfrak{s o}}(p,q-1))$
 for \lq\lq{$p \ge q \ge 4$ and $p\equiv q \mod 2$}\rq\rq, 
 \lq\lq{$p \ge 5$ and $q=2$}\rq\rq,  
 or 
 \lq\lq{$p \ge 4$ and $q=3$}\rq\rq.  
\par\noindent
$\bullet$\enspace
$({\mathfrak{f}}_{4(4)}, {\mathfrak{s o}}(5,4))$, 
\par\noindent
$\bullet$\enspace
$({\mathfrak{e}}_{6(6)}, {\mathfrak{f}}_{4(4)})$, 
 or 
$({\mathfrak{e}}_{6(-14)},{\mathfrak{f}}_{4(-20)})$.  

\end{example}

We note 
 that the upper bound $m(\Pi|_{G'})$ or $m(\Pi_1 \otimes \Pi_2)$ 
 of the multiplicity
 can be larger than 1 in Theorems \ref{thm:Joseph} and \ref{thm:tensormin}, 
 see {\it{e.g.}} \cite{KOP}
 for an explicit branching law of the restriction $\Pi|_{G'}$
 when $(G,G')=(SL(n,{\mathbb{R}}), SO(p,q))$ with $p+q=n$.  
However, 
 it is plausible
 that a multiplicity-free theorem holds 
 in Theorems \ref{thm:minbdd} and \ref{thm:tensor}:
\begin{conjecture}
\label{conj:multone}
$m(\Pi|_{G'})=1$
 in Theorem \ref{thm:minbdd}, 
 and $m(\Pi_1 \otimes \Pi_2) =1$ in Theorem \ref{thm:tensor}.  
\end{conjecture}

Conjecture \ref{conj:multone} holds 
 when $(G,G')$ is a Riemannian symmetric pair $(G,K)$, 
 see \cite[Prop.\ 4.10]{GS05}.

\begin{remark}
(1)\enspace
The Joseph ideal is not defined 
for ${\mathfrak{sl}}(n,{\mathbb{C}})$, 
 hence there is no minimal representation
 in the sense of Definition \ref{def:minrep}
 for $G=SL(n,{\mathbb{R}})$, 
 for instance.  
However there exist continuously many $\Pi \in \operatorname{Irr}(G)$ 
({\it{e.g.}}, degenerate principal series representations
 induced from a mirabolic subgroup)
 for $G=SL(n,{\mathbb{R}})$
 such that $\operatorname{DIM}(\Pi)=n({\mathfrak{g}}_{\mathbb{C}})$, 
 and Theorems \ref{thm:Joseph} and \ref{thm:tensormin} apply
 to these representations.  
The Plancherel-type theorem for the restriction $\Pi|_{G'}$
 is proved in \cite{KOP} 
 for all symmetric pairs $(G,G')$
 when $\Pi$ is a {\it{unitarily}} induced representation.  
See also Example \ref{ex:BC} below.  
\par\noindent
(2)\enspace
The inequality \eqref{eqn:GKbdd} depends only 
 on the complexification ${\mathfrak{g}}_{\mathbb{C}}$, 
 and is not necessarily optimal 
 for specific real forms ${\mathfrak{g}}$.  
In fact, 
 one has a better inequality 
 $n({\mathfrak{g}}) \le \operatorname{DIM}(\Pi)$
 where $n({\mathfrak{g}})$ depends on the real form ${\mathfrak{g}}$, 
 see Section \ref{subsec:realmin}.  
For most of real Lie algebras
 one has $n({\mathfrak{g}})=n({\mathfrak{g}}_{\mathbb{C}})$, 
 but there are a few simple Lie algebras ${\mathfrak{g}}$
satisfying $n({\mathfrak{g}})>n({\mathfrak{g}}_{\mathbb{C}})$.  
For example, 
 if $G=Sp(p,q)$, 
 $n({\mathfrak{g}})=2(p+q)-1 > n({\mathfrak{g}}_{\mathbb{C}}) = p+q$, 
 hence there is no $\Pi \in \operatorname{Irr}(G)$
 with $\operatorname{DIM}(\Pi)=n({\mathfrak{g}}_{\mathbb{C}})$, 
 however, 
 there exists a countable family 
 of $\Pi \in \operatorname{Irr}(G)$
 with $\operatorname{DIM}(\Pi)=n({\mathfrak{g}})$, 
 to which another bounded multiplicity theorem 
(Theorem \ref{thm:Rminbdd} in Section \ref{sec:coiso})
 applies.  
\par\noindent
(3)\enspace
Concerning Theorem \ref{thm:Joseph}, 
 the bounded property
 of the multiplicity
 in the tensor product representations $\Pi_1 \otimes \Pi_2$ still holds
 for some other \lq\lq{small representations}\rq\rq\
 $\Pi_1$ and $\Pi_2$
 whose Gelfand--Kirillov dimensions are greater
 than $n({\mathfrak{g}}_{\mathbb{C}})$.  
See \cite[Thm.\ 1.5 and Cor.\ 4.10]{K22}
 for example.  
\end{remark}

This paper is organized 
 as follows.  
In Section \ref{sec:2}
 we give a brief review 
 of some background of the problem, 
 examples, 
 and known theorems.  
Section \ref{sec:coiso} is devoted to the proof
 of Theorems \ref{thm:Joseph}, \ref{thm:tensormin} and \ref{thm:irr}.

\vskip 1pc
\par\noindent
{\bf{$\langle$Acknowledgements$\rangle$}}\enspace
The author warmly thanks Professor Vladimir Dobrev for his
hospitality during the 14th International Workshop: Lie Theory and its 
Applications in Physics, held online in Bulgaria, 20--26 June 2021.
This work was partially supported
 by Grant-in-Aid for Scientific Research (A) (18H03669), 
JSPS.

\section{Background and motivation}
\label{sec:2}

In this section, 
 we explain some background, 
 examples, 
 and known theorems 
 in relation to our main results.

If $\Pi$ is an irreducible {\it{unitary}} representation
 of a group $G$, 
 then one may consider the irreducible decomposition 
 ({\it{branching law}})
 of the restriction $\Pi|_{G'}$
 to a subgroup $G'$
 by using the direct integral of Hilbert spaces.  
For non-unitary representations $\Pi$, 
 such an irreducible decomposition does not make sense, 
 but the computation
 of the multiplicity $[\Pi|_{G'}:\pi]$
 for all $\pi \in {\operatorname{Irr}}(G')$
 may be
 thought of as a variant of branching laws.  
Here we recall from Section \ref{sec:Intro}
 that for $\Pi \in {\mathcal{M}}(G)$
 and $\pi \in {\operatorname{Irr}}(G')$
 that the multiplicity 
 $[\Pi|_{G'}:\pi]$ is the dimension
 of the space $\invHom {G'}{\Pi|_{G'}}{\pi}$
 of symmetry breaking operators.

By branching problems in representation theory, 
 we mean the broad problem of understanding
 how irreducible (not necessarily, unitary) representations
 of a group behave when restricted to a subgroup.  
As viewed in \cite{xKVogan2015}, 
 we may divide the branching problems
 into the following three stages:
\par\noindent
{\bf{Stage A.}}\enspace
Abstract features of the restriction;
\par\noindent
{\bf{Stage B.}}\enspace
Branching law;
\par\noindent
{\bf{Stage C.}}\enspace
Construction of symmetry breaking operators.  

The role of Stage A is to develop 
 a theory on the restriction of representations
 as generally
 as possible.  
In turn, 
 we may expect a detailed study of the restriction
 in Stages B (decomposition of representations)
 and C (decomposition of vectors)
 in the \lq\lq{promising}\rq\rq\ settings 
 that are suggested by the general theory 
 in Stage A.

The study of the upper estimate of the multiplicity in this article
 is considered as a question 
 in Stage A
 of branching problems.

For a detailed analysis on the restriction $\Pi|_{G'}$
 in Stages B and C, 
 it is desirable to have
 the bounded multiplicity property
 $m(\Pi|_{G'})<\infty$
 (see Definition \ref{def:bdd}), 
 or at least 
 to have the finite multiplicity property
\begin{equation}
\label{eqn:fweak}
     [\Pi|_{G'}: \pi]<\infty
     \quad
     \text{for $\pi \in \operatorname{Irr}(G')$.}
\end{equation}
In the previous papers \cite{xkAnn98, K14, K21Kostant, K22, K22PJA, xktoshima}
 we proved some general theorems 
 for bounded/finite multiplicities
 of the restriction $\Pi|_{G'}$, 
 which we review briefly now.

\subsection{Bounded multiplicity pairs $(G,K')$ with $K'$ compact}

Harish-Chandra's admissibility theorem tells
 the finiteness property \eqref{eqn:fweak} holds
 for any $\Pi \in {\mathcal{M}}(G)$
 if $G'$ is a maximal compact subgroup $K$ of $G$.  
More generally, 
 the finiteness property \eqref{eqn:fweak} for a compact subgroup plays
 a crucial role
 in the study of discretely decomposable restriction
 with respect to reductive subgroups
 \cite{xkInvent94, xkAnn98, xkInvent98, K21Kostant}.  
We review briefly the necessary and sufficient condition 
 for \eqref{eqn:fweak}
 when $G'$ is compact.  
In this subsection, 
 we use the letter $K'$
 instead of $G'$
 to emphasize that $G'$ is compact.  
Without loss of generality, 
 we may and do assume
 that $K'$ is contained in $K$.

\begin{fact}
[{\cite{xkAnn98, K21Kostant}}]
\label{fact:Kadm}
Suppose that $K'$ is a compact subgroup
 of a real reductive group $G$.  
Let $\Pi \in {\mathcal{M}}(G)$.  
Then the following two conditions on the triple $(G, K', \Pi)$
 are equivalent:
\par\noindent
(i)\enspace
The finite multiplicity property \eqref{eqn:fweak} holds.  
\par\noindent
(ii)\enspace
$\operatorname{AS}_K(\Pi) \cap C_K(K') =\{0\}$.  
\end{fact}

Here $\operatorname{AS}_K(\Pi)$ is the asymptotic $K$-support
 of $\Pi$, 
 and $C_K(K')$ is the momentum set
 for the natural action 
 on the cotangent bundle $T^{\ast}(K/K')$.  
There are two proofs
 for the implication (ii) $\Rightarrow$ (i):
 by using the singularity spectrum
 (or the wave front set)
 \cite{xkAnn98}
 and by using symplectic geometry \cite{K21Kostant}.  
The proof for the implication (i) $\Rightarrow$ (ii) is given
 in \cite{K21Kostant}.  
See \cite{KO15} for some classification theory.  
\subsection{Bounded/finite multiplicity pairs $(G,G')$}
\label{subsec:bdd}

We now consider the general case
 where $G'$ is not necessarily compact.  
In \cite{K14} and \cite[Thms.~C and D]{xktoshima}
 we proved the following geometric criteria
 that concern {\it{all}} $\Pi \in {\operatorname{Irr}}(G)$
 and {\it{all}} $\pi \in {\operatorname{Irr}}(G')$:
\begin{fact}
\label{fact:KO}
Let $G \supset G'$ be a pair of real reductive algebraic Lie groups.  
\par\noindent
(1)\enspace
{\bf{Bounded multiplicity}} for a pair $(G,G')$: \enspace
\begin{equation}
\label{eqn:BB}
  \underset{\Pi \in \operatorname{Irr}(G)}\sup\,\,
  \underset{\pi \in \operatorname{Irr}(G')}\sup\,\,
  [\Pi|_{G'}:\pi]<\infty
\end{equation}
if and only if $(G_{\mathbb{C}} \times G_{\mathbb{C}}')/\operatorname{diag} G_{\mathbb{C}}'$ is spherical.   
\par\noindent
(2)\enspace
{\bf{Finite multiplicity}} for a pair $(G,G')$:\enspace
\[
  [\Pi|_{G'}:\pi]<\infty, 
\quad
 {}^{\forall}\Pi \in \operatorname{Irr}(G), 
 {}^{\forall}\pi \in \operatorname{Irr}(G')
\]
if and only if $(G \times G')/\operatorname{diag}G'$ is real spherical.  
\end{fact}

Here we recall 
 that a complex $G_{\mathbb{C}}$-manifold $X$
 is called {\it{spherical}}
 if a Borel subgroup of $G_{\mathbb{C}}$
 has an open orbit in $X$, 
 and that a $G$-manifold $Y$ is called
 {\it{real spherical}}
 if a minimal parabolic subgroup of $G$
 has an open orbit in $Y$.

A remarkable discovery in Fact \ref{fact:KO} (1) was 
 that the bounded multiplicity property \eqref{eqn:BB} is determined
 only by the complexified Lie algebras
 ${\mathfrak{g}}_{\mathbb{C}}$ and ${\mathfrak{g}}_{\mathbb{C}}'$.  
In particular, 
 the classification
 of such pairs $(G,G')$
 is very simple, 
 because it is reduced to a classical result 
 when $G$ is compact \cite{xkramer}:
the pair $({\mathfrak{g}}_{\mathbb{C}}, {\mathfrak{g}}_{\mathbb{C}}')$ is the direct sum of the following ones
up to abelian ideals:
\begin{equation}
\label{eqn:BBlist}
({\mathfrak{sl}}_n, {\mathfrak{gl}}_{n-1}), 
({\mathfrak{so}}_{n}, {\mathfrak{so}}_{n-1}),
\text{ or } 
({\mathfrak{so}}_8, {\mathfrak{spin}}_7). 
\end{equation}

See \cite{KS15, xksbonvec}
{\it{e.g.}}, 
 for some recent developments
 in Stage C
 such as detailed analysis on symmetry breaking operators
 for some non-compact real forms of the pairs \eqref{eqn:BBlist}.

On the other hand, 
 the finite multiplicity property in Fact \ref{fact:KO} (2)
 depends on real forms $G$ and $G'$.  
It is fulfilled 
 for any Riemannian symmetric pair, 
 which is Harish-Chandra's admissibility theorem.  
More generally for non-compact $G'$, 
 the finite-multiplicity property \eqref{eqn:fweak} often holds
 when the restriction $\Pi|_{G'}$ decomposes discretely, 
 see \cite{xkInvent94, xkAnn98, xkInvent98} 
 for the general theory of \lq\lq{$G'$-admissible restriction}\rq\rq.  
However, 
 for some reductive symmetric pairs
 such as $(G,G')=(S L(p+q,{\mathbb{R}}), S O(p,q))$, 
 there exists $\Pi \in \operatorname{Irr}(G)$
 for which the finite multiplicity property \eqref{eqn:fweak}
 of the restriction $\Pi|_{G'}$ fails, 
 as we have seen in Example \ref{ex:fmult}.  
Such $\Pi$ is fairly \lq\lq{large}\rq\rq.

\subsection{Uniform estimates for a family of small representations}
The classification in \cite{xKMt}
 tells that the class of the reductive symmetric pairs $(G,G')$ 
 satisfying the finite multiplicity property \eqref{eqn:PP}
 is much broader than that of real forms $(G,G')$
 corresponding to those complex pairs in (5).  
However, 
 there also exist pairs $(G,G')$
 beyond the list of \cite{xKMt}
 for which we can still expect fruitful branching laws 
 of the restriction $\Pi|_{G'}$
 in Stages B and C
 for some $\Pi \in \operatorname{Irr}(G)$.  
Such $\Pi$ must be a  \lq\lq{small representation}\rq\rq.  
Here are some known examples:
\begin{example}
\label{ex:BC}
(1)\enspace({\bf{Stage B}})\enspace
See-saw dual pairs
 (\cite{Ku})
 yield explicit formul{\ae}
 of the multiplicity
 for the restriction of small representations, 
 with respect to some classical symmetric pairs
 $(G,G')$.  
\par\noindent
(2)\enspace ({\bf{Stage C}})\enspace
For $G=SL(n,{\mathbb{R}})$, 
 any degenerate representation $\Pi=\operatorname{Ind}_P^G({\mathbb{C}}_{\lambda})$
 induced from a \lq\lq{mirabolic subgroup}\rq\rq\ $P$ of $G$
 has the smallest Gelfand--Kirillov dimension
 $n({\mathfrak{g}}_{\mathbb{C}})$.  
For a unitary character ${\mathbb{C}}_{\lambda}$, 
 the Plancherel-type formula
 of the restriction $\Pi|_{G'}$
 is determined in \cite{KOP}
 for {\it{all}} symmetric pairs $(G,G')$.  
The feature of the restriction $\Pi|_{G'}$ is summarized
 as follows:
let $p+q=n$, 
 and when $n$ is even we write $n=2m$.  
\newline\noindent
$\cdot$
$G'=S(GL(p,{\mathbb{R}}) \times GL(q,{\mathbb{R}}))$.  
\newline\noindent
\hphantom{G'}
$\cdots$
Only continuous spectrum appears with multiplicity one.  
\newline\noindent
$\cdot$
$G'=SL(m,{\mathbb{C}}) \cdot {\mathbb{T}}$.  
\newline\noindent
\hphantom{G'}
$\cdots$
Only discrete spectrum appears with multiplicity one.  
\newline\noindent
$\cdot$
$G'=SO(p,q)$.  
\newline\noindent
\hphantom{G'}
$\cdots$
Discrete spectrum appears with multiplicity one, 
\newline\noindent
\hphantom{G'$\cdots$}
and continuous spectrum appears with multiplicity two.  
\newline\noindent
$\cdot$
$G'=Sp(m,{\mathbb{R}})$
\newline\noindent
\hphantom{G'}
$\cdots$ Almost irreducible (See also Theorem \ref{thm:irr}).  
\par
The uniform bounded multiplicity property
 in all these cases
 ({\bf{Stage A}})
 is guaranteed by Theorem \ref{thm:Joseph} in this article
 because $\operatorname{DIM}(\Pi)$ attains $n({\mathfrak{g}}_{\mathbb{C}})$, 
 and alternatively, 
 by another general result \cite[Thm.\ 4.2]{K22}.  
\par\noindent
(3)\enspace({\bf{Stage C}})\enspace
For the symmetric pair $(G,G')=(O(p,q), O(p_1,q_1) \times O(p_2,q_2))$
 with $p_1+p_2=p$ and $q_1+q_2=q$, 
 by using the Yamabe operator in conformal geometry, 
 discrete spectrum
 in the restriction $\Pi|_{G'}$
 of the minimal representation $\Pi$
 was obtained geometrically in \cite{KOr}.  
Moreover, 
 for the same pair $(G,G')$, 
 discrete spectrum in the restriction $\Pi|_{G'}$ was explicitly constructed
 and classified when $\Pi$ 
 belongs to cohomologically parabolic induced representation
 $A_{\mathfrak{q}}(\lambda)$ from
 a maximal $\theta$-stable parabolic subalgebra ${\mathfrak{q}}$
 in \cite{K21}.  
In contrast to Example \ref{ex:fmult} (2), 
 the multiplicity is one
 for any $p_1$, $q_1$, $p_2$, and $q_2$.  
\end{example}

In view of these nice cases, 
 and also in search for further broader settings 
 in which we could expect a detailed study of the restriction $\Pi|_{G'}$
 in Stages B and C, 
 we addressed the following:
\begin{problem}
[{\cite[Prob.\ 6.2]{xKVogan2015}}, {\cite[Prob.\ 1.1]{K22}}]
\label{q:Bdd}
Given a pair $G \supset G'$, 
 find a subset $\Omega$
 of ${\mathcal{M}}(G)$
 such that 
$
  \underset{\Pi \in \Omega} \sup \,\, m(\Pi|_{G'})<\infty.  
$
\end{problem}

Since branching problems often arise
 for a family of representations $\Pi$, 
 the formulation of Problem \ref{q:Bdd}
 is to work with the {\it{triple}} $(G,G',\Omega)$
 rather than the pair $(G,G')$ for the finer study of multiplicity estimates
 of the restriction $\Pi|_{G'}$.  
Fact \ref{fact:KO} (1) deals with the case
 $\Omega=\operatorname{Irr}(G)$.  
In \cite{K21, K22}, 
 we have considered Problem \ref{q:Bdd}
 including the following cases:
\begin{enumerate}
\item[(1)]
 $\Omega=\operatorname{Irr}(G)_H$, 
  the set of $H$-distinguished 
 irreducible representations of $G$
 where $(G,H)$ is a reductive symmetric pair;
\item[(2)]
$\Omega=\Omega_P$,
the set of induced representations from characters
 of a parabolic subgroup $P$ of $G$;
\item[(3)]
$\Omega=\Omega_{P,{\mathfrak{q}}}$,
 certain families of (vector-bundle valued) degenerate principal series representations.  
\end{enumerate}

For the readers' convenience, 
 we give a flavor
 of the solutions to Problem \ref{q:Bdd}
 in the above cases by quoting the criteria from \cite{K22}.  
See \cite{K22PJA} for a brief survey.

We write $G_{\mathbb{C}}$
 for the complexified Lie group $G$, 
 and $G_U$ for the compact real form of $G_{\mathbb{C}}$.  
For a reductive symmetric pair $(G,H)$, 
 one can define a Borel subgroup $B_{G/H}$
 which is a parabolic subgroup in $G_{\mathbb{C}}$, 
 see \cite[Def.\ 3.1]{K22PJA}.  
Note that $B_{G/H}$ is not necessarily solvable.  
For $\Omega=\operatorname{Irr}(G)_H$
 when $(G,H)$ is a reductive symmetric pair, 
one has the following answer to Problem \ref{q:Bdd}:

\begin{fact}
[{\cite[Thm.\ 1.4]{K22}}]
\label{fact:bdd}
Let $B_{G/H}$ be a Borel subgroup for $G/H$.  
Suppose $G'$ is an algebraic reductive subgroup of $G$.  
Then the following three conditions on the triple $(G,H,G')$
 are equivalent:
\begin{enumerate}
\item[{\rm{(i)}}]
$\underset{\Pi \in \operatorname{Irr}(G)_H}\sup m(\Pi|_{G'}) <\infty$.  

\item[{\rm{(ii)}}]
$G_{\mathbb{C}}/B_{G/H}$ is $G_U'$-strongly visible.  

\item[{\rm{(iii)}}]
$G_{\mathbb{C}}/B_{G/H}$ is $G_{\mathbb{C}}'$-spherical.  
\end{enumerate}
\end{fact}

For $\Omega=\Omega_P$, 
 one has the following answer to Problem \ref{q:Bdd}:
\begin{fact}
[{\cite[Ex.\ 4.5]{K22}}, {\cite{Tu}}]
\label{fact:introQsph2}
Let $G \supset G'$ be a pair of real reductive algebraic Lie groups, 
 and $P$ a parabolic subgroup of $G$.  
Then one has the equivalence on the triple $(G,G';P):$
\par\noindent
(i)\enspace
$\underset{\Pi \in \Omega_P}\sup m(\Pi|_{G'})< \infty$.  
\par\noindent
(ii)\enspace
$G_{\mathbb{C}}/P_{\mathbb{C}}$ is strongly $G_U'$-visible.  
\par\noindent
(iii)\enspace
$G_{\mathbb{C}}/P_{\mathbb{C}}$ is $G_{\mathbb{C}}'$-spherical.  
\end{fact}

The following is a useful extension of Fact \ref{fact:introQsph2}.  
\begin{fact}
[{\cite[Thm.\ 4.2]{K22}}]
\label{fact:Qsph}
Let $G \supset G'$ be a pair of real reductive algebraic Lie groups, 
 $P$ a parabolic subgroup of $G$, 
 and $Q$ a complex parabolic subgroup of $G_{\mathbb{C}}$
 such that ${\mathfrak{q}} \subset {\mathfrak{p}}_{\mathbb{C}}$.  
One defines a subset $\Omega_{P,{\mathfrak{q}}}$ in ${\mathcal{M}}(G)$
 that contains $\Omega_P$
 (see \cite{K22} for details).  
Then the following three conditions 
 on $(G, G';P, Q)$ are equivalent:
\begin{enumerate}
\item[{\rm{(i)}}]
$
  \underset{\Pi \in \Omega_{P,{\mathfrak{q}}}}\sup\,\,
  m(\Pi|_{G'})
   <\infty.  
$

\item[{\rm{(ii)}}]
$G_{\mathbb{C}}/Q$ is $G_U'$-strongly visible.  

\item[{\rm{(iii)}}]
$G_{\mathbb{C}}/Q$ is $G_{\mathbb{C}}'$-spherical.  
\end{enumerate}
\end{fact}

These criteria lead us to classification results
 for the triples $(G,G',\Omega)$, 
 see \cite{K21, K22, K22PJA}
 and references therein.

The representations $\Pi$ in $\Omega = \operatorname{Irr}(G)_H$
 or $\Omega_P$, $\Omega_{P,{\mathfrak{q}}}$
 are fairly small, 
however, 
 the classification results in \cite{K22} 
 indicate that some symmetric pairs $(G,G')$ still do not appear
 for such a family $\Omega$.  
A clear distinction from these previous results is
 that Theorem \ref{thm:Joseph} allows {\it{all}} symmetric pairs $(G,G')$
 for an affirmative answer to Problem \ref{q:Bdd}
 in the extremal case
where $\Omega=\{\Pi\}$
 with $\operatorname{DIM}(\Pi)=n({\mathfrak{g}}_{\mathbb{C}})$.

Concerning the method of the proof, 
 we utilized in \cite{xktoshima} 
 hyperfunction boundary maps 
 for the \lq\lq{if}\rq\rq\ part
 ({\it{i.e.,}} the sufficiency of the bounded multiplicity property)
 and a generalized Poisson transform \cite{K14}
 for the \lq\lq{only if}\rq\rq\ part
 in the proof of Fact \ref{fact:KO}.  
The proof in \cite{K22, Tu} used a theory of holonomic ${\mathcal{D}}$-modules
 for the \lq\lq{if}\rq\rq\ part.  
Our proof in this article still uses
 a theory of ${\mathcal{D}}$-modules, 
 and more precisely, 
 the following:

\begin{fact}
[{\cite{Ki}}]
\label{fact:Ki}
Let ${\mathcal{I}}$ be the annihilator
 of $\Pi \in {\mathcal{M}}(G)$
 in the enveloping algebra $U({\mathfrak{g}}_{\mathbb{C}})$.  
Assume that the $G_{\mathbb{C}}'$-action
 on the associated variety of ${\mathcal{I}}$
 is coisotropic
 (Definition \ref{def:coisotropic}).  
Then the restriction $\Pi|_{G'}$ has the bounded multiplicity property
 (Definition \ref{def:bdd}).  
\end{fact}

We note
 that the assumption in Fact \ref{fact:Ki} depends
 only on the complexification
 of the pair $({\mathfrak{g}}, {\mathfrak{g}}')$
 of the Lie algebras.  
Thus the proof of Theorems \ref{thm:Joseph} and \ref{thm:tensormin}
 is reduced to a geometric question
 on {\it{holomorphic}} coisotropic actions
 on {\it{complex}} nilpotent coadjoint orbits, 
 which will be proved in Theorem \ref{thm:22030921}.

\section{Coisotropic action on coadjoint orbits}
\label{sec:coiso}
Let $V$ be a vector space
 endowed with a symplectic form $\omega$.  
A subspace $W$ is called {\it{coisotropic}}
 if $W^{\perp} \subset W$, 
 where 
\[
   W^{\perp}:=\{v \in V: \text{$\omega(v,\cdot)$ vanishes on $W$}\}.  
\]

The concept of coisotropic actions is defined infinitesimally as follows.  
\begin{definition}
[Huckleberry--Wurzbacher {\cite{huwu90}}]
\label{def:coisotropic}
Let $H$ be a connected Lie group, 
 and $X$ a Hamiltonian $H$-manifold.  
The $H$-action is called
 {\it{coisotropic}}
 if there is an $H$-stable open dense subset $U$ of $X$
 such that $T_x(H \cdot x)$ is a coisotropic subspace
 in the tangent space $T_x X$ for all $x \in U$.  
\end{definition}

Any coadjoint orbit of a Lie group $G$
 is a Hamiltonian $G$-manifold
 with the Kirillov--Kostant--Souriau symplectic form.  
The main result of this section is the following:

\begin{theorem}
\label{thm:22030921}
Let ${\mathbb{O}}_{\operatorname{min},{\mathbb{C}}}$ be the minimal nilpotent coadjoint orbit of a connected complex simple Lie group $G_{\mathbb{C}}$.  
\begin{enumerate}
\item[{\rm{1)}}]
The diagonal action of $G_{\mathbb{C}}$
 on 
$
   {\mathbb{O}}_{\operatorname{min},{\mathbb{C}}} \times {\mathbb{O}}_{\operatorname{min},{\mathbb{C}}}
$
 is coisotropic.  
\item[{\rm{2)}}]
For any symmetric pair $(G_{\mathbb{C}}, K_{\mathbb{C}})$, 
 the $K_{\mathbb{C}}$-action 
 on ${\mathbb{O}}_{\operatorname{min},{\mathbb{C}}}$
 is coisotropic.  
\end{enumerate}
\end{theorem}

\subsection{Generalities: coisotropic actions on coadjoint orbits}
\label{subsec:gencoiso}
We begin with a general setting
 for a {\it{real}} Lie group.  
Suppose that ${\mathbb{O}}$ is a coadjoint orbit
 of a connected Lie group $G$
 through $\lambda \in {\mathfrak{g}}^{\ast}$.  
Denote by $G_{\lambda}$ the stabilizer subgroup of $\lambda$ in $G$, 
 and by $Z_{\mathfrak{g}}(\lambda)$ its Lie algebra.  
Then the Kirillov--Kostant--Souriau symplectic form $\omega$
 on the coadjoint orbit
 ${\mathbb{O}}= \operatorname{Ad}^{\ast}(G) \simeq G/G_{\lambda}$
 is given at the tangent space
 $T_{\lambda} {\mathbb{O}} \simeq {\mathfrak{g}}/Z_{\mathfrak{g}}(\lambda)$
 by
\begin{equation}
\label{eqn:KKS}
\omega \colon 
{\mathfrak{g}}/Z_{\mathfrak{g}}(\lambda) \times {\mathfrak{g}}/Z_{\mathfrak{g}}(\lambda)
\to
{\mathbb{R}}, 
\quad
(X,Y) \mapsto \lambda([X,Y]).  
\end{equation}
Suppose that $H$ is a connected subgroup 
 with Lie algebra ${\mathfrak{h}}$.  
For $\lambda \in {\mathfrak{g}}^{\ast}$, 
 we define a subspace of the Lie algebra ${\mathfrak{g}}$ by 
\begin{equation}
\label{eqn:zhlmd}
Z_{\mathfrak{g}}({\mathfrak{h}};\lambda)
:=
\{Y \in {\mathfrak{g}}: \text{$\lambda([X,Y])=0$
 for all $X \in {\mathfrak{h}}$}
\}.  
\end{equation}
Clearly,
 $Z_{\mathfrak{g}}({\mathfrak{h}};\lambda)$ contains
 the Lie algebra $Z_{\mathfrak{g}}(\lambda)\equiv Z_{\mathfrak{g}}({\mathfrak{g}};\lambda)$
 of $G_{\lambda}$.

We shall use the following:
\begin{lemma}
\label{lem:1.7}
The $H$-action on a coadjoint orbit ${\mathbb{O}}$
 in ${\mathfrak{g}}^{\ast}$
 is coisotropic 
 if there exists a subset $S$ (slice)
 in ${\mathbb{O}}$
 with the following two properties:
\begin{align}
&\text{$\operatorname{Ad}^{\ast}(H) S$ is open dense in ${\mathbb{O}}$}, 
\notag
\\
\label{eqn:coiso}
&\text{$Z_{\mathfrak{g}}({\mathfrak{h}};\lambda)
\subset {\mathfrak{h}}+Z_{\mathfrak{g}}(\lambda)$ 
 for any $\lambda \in S$.}
\end{align}
\end{lemma}

\begin{proof}
It suffices to verify
 that $T_{\lambda}(\operatorname{Ad}^{\ast}(H) \lambda)$
 is a coisotropic subspace
 in $T_{\lambda}{\mathbb{O}}$
 for any $\lambda \in S$
 because the condition \eqref{eqn:coiso} is $H$-invariant.  
Via the identification 
 $T_{\lambda} {\mathbb{O}}
 \simeq {\mathfrak{g}}/Z_{\mathfrak{g}}(\lambda)$, 
 one has
$T_{\lambda}(\operatorname{Ad}^{\ast}(H)\lambda)
\simeq
({\mathfrak{h}}+Z_{\mathfrak{g}}(\lambda))/
Z_{\mathfrak{g}}(\lambda)$.  
By the formula \eqref{eqn:KKS} of the symplectic form $\omega$
 on ${\mathbb{O}}$, 
one has 
$
   T_{\lambda}(\operatorname{Ad}^{\ast}(H)\lambda)^{\perp}
   \simeq
   Z_{\mathfrak{g}}({\mathfrak{h}};\lambda)/Z_{\mathfrak{g}}(\lambda)
$.  
Hence 
$
T_{\lambda}(\operatorname{Ad}^{\ast}(H)\lambda)
$ is a coisotropic subspace in $T_{\lambda}{\mathbb{O}}$
 if and only if 
$
   Z_{\mathfrak{g}}({\mathfrak{h}};\lambda)
   \subset
   {\mathfrak{h}}+Z_{\mathfrak{g}}(\lambda)
$, 
 whence the lemma.  
\end{proof}

For semisimple ${\mathfrak{g}}$, 
 the Killing form $B$ induces the following $G$-isomorphism
\begin{equation}
\label{eqn:Xlmd}
{\mathfrak{g}}^{\ast} \simeq {\mathfrak{g}}, 
\quad
 \lambda \mapsto X_{\lambda}.  
\end{equation}
By definition, 
 one has 
$
\lambda([X,Y])=B(X_{\lambda}, [X,Y])=B([X_{\lambda}, X],Y)
$, 
 and thus
\[
   Z_{\mathfrak{g}}({\mathfrak{h}};\lambda)
   = [X_{\lambda}, {\mathfrak{h}}]^{\perp B}, 
\]
 where the right-hand side stands for the orthogonal complement subspace
 of $[X_{\lambda}, {\mathfrak{h}}]:=\{[X_{\lambda}, X]:X \in {\mathfrak{h}}\}$
 in ${\mathfrak{g}}$
 with respect to the Killing form $B$.  
Hence we have the following.  
\begin{lemma}
\label{lem:Hcoiso}
For semisimple ${\mathfrak{g}}$, 
 one may replace 
 the condition \eqref{eqn:coiso}
 in Lemma \ref{lem:1.7}
 by 
\begin{equation}
\label{eqn:coisoB}
({\mathfrak{h}}+Z_{\mathfrak{g}}(\lambda))^{\perp B}
\subset 
[X_{\lambda}, {\mathfrak{h}}]
\quad
\text{for any $\lambda$.  }
\end{equation}
\end{lemma}

\subsection{Real minimal nilpotent orbits}
\label{subsec:realmin}

Let $G$ be a connected non-compact simple Lie group
 without complex structure.  
Denote by ${\mathcal{N}}$ the nilpotent cone
 in ${\mathfrak{g}}$, 
 and ${\mathcal{N}}/G$ the set of nilpotent orbits, 
 which may be identified with nilpotent coadjoint orbits 
 in ${\mathfrak{g}}^{\ast}$ via \eqref{eqn:Xlmd}.  
The finite set ${\mathcal{N}}/G$ is a poset 
 with respect to the closure ordering, 
 and there are at most two minimal elements
 in $({\mathcal{N}} \setminus \{0\})/G$, 
 which we refer to as {\it{real minimal nilpotent (coadjoint) orbits}}.  
See \cite{B98, KO15, O15} for details.  
The relationship with the complex minimal nilpotent orbits
 ${\mathbb{O}}_{\operatorname{min}, {\mathbb{C}}}$
 in the complexified Lie algebra 
$
   {\mathfrak{g}}_{\mathbb{C}}:= {\mathfrak{g}}\otimes_{\mathbb{R}}{\mathbb{C}}
$
 is given as below.  
Let $K$ be a maximal compact subgroup of $G$ modulo center.

\begin{lemma}
\label{lem:CRmin}
In the setting above, 
 exactly one of the following cases occurs.  
\begin{enumerate}
\item[{\rm{(1)}}]
$({\mathfrak{g}}, {\mathfrak{k}})$ is not of Hermitian type, 
 and ${\mathbb{O}}_{\operatorname{min}, {\mathbb{C}}} \cap {\mathfrak{g}}
=\emptyset$.  

\item[{\rm{(2)}}]
$({\mathfrak{g}}, {\mathfrak{k}})$ is not of Hermitian type, 
 and ${\mathbb{O}}_{\operatorname{min}, {\mathbb{C}}} \cap {\mathfrak{g}}$
 is a single orbit of $G$.  

\item[{\rm{(3)}}]
$({\mathfrak{g}}, {\mathfrak{k}})$ is of Hermitian type, 
 and ${\mathbb{O}}_{\operatorname{min}, {\mathbb{C}}} \cap {\mathfrak{g}}$
 consists of two orbits of $G$.  
\end{enumerate}
\end{lemma}

As the $G$-orbit decomposition
 of ${\mathbb{O}}_{\operatorname{min}, {\mathbb{C}}} \cap {\mathfrak{g}}$, 
 we write
$
   {\mathbb{O}}_{\operatorname{min}, {\mathbb{C}}} \cap {\mathfrak{g}}
= \{ {\mathbb{O}}_{\operatorname{min}, {\mathbb{R}}}\}
$ in Case (2), 
$
    {\mathbb{O}}_{\operatorname{min}, {\mathbb{C}}} \cap {\mathfrak{g}}
   = \{ {\mathbb{O}}_{\operatorname{min}, {\mathbb{R}}}^+, {\mathbb{O}}_{\operatorname{min}, {\mathbb{R}}}^-\}$ in Case (3).  
Then they exhaust
 all real minimal nilpotent orbits
 in Cases (2) and (3).  
Real minimal nilpotent orbits are unique in Case (1), 
 to be denoted by ${\mathbb{O}}_{\operatorname{min}, {\mathbb{R}}}$.  
We set
\begin{equation}
\label{eqn:ng}
 n({\mathfrak{g}}):=
\begin{cases}
\frac 1 2 \dim {\mathbb{O}}_{\operatorname{min}, {\mathbb{R}}}
\quad
&\text{in Cases (1) and (2)}, 
\\
\frac 1 2 \dim {\mathbb{O}}_{\operatorname{min}, {\mathbb{R}}}^+
=\frac 1 2 \dim {\mathbb{O}}_{\operatorname{min}, {\mathbb{R}}}^-
\quad
&\text{in Case (3)}.  
\end{cases}
\end{equation}
Then $n({\mathfrak{g}})=n({\mathfrak{g}}_{\mathbb{C}})$
 in Cases (2) and (3), 
 and $n({\mathfrak{g}})>n({\mathfrak{g}}_{\mathbb{C}})$ in Case (1).  
The formula of $n({\mathfrak{g}})$ in Case (1) is given in 
 \cite{O15} as follows.  

\begin{figure}[H]
\begin{center}
\begin{tabular}{c|ccccc}
${\mathfrak{g}}$
& ${\mathfrak{s u}}^{\ast}(2n)$
& ${\mathfrak{s o}}(n-1,1)$
& ${\mathfrak{s p}}(m,n)$
& ${\mathfrak{f}}_{4(-20)}$
& ${\mathfrak{e}}_{6(-26)}$
\\
\hline
$n({\mathfrak{g}})$
& $4n-4$
& $n-2$
& $2(m+n)-1$
& $11$
& $16$
\\
\end{tabular}
\end{center}
\end{figure}

For any $\Pi \in \operatorname{Irr}(G)$, 
 the Gelfand--Kirillov dimension 
 $\operatorname{DIM}(\Pi)$ satisfies
 $n({\mathfrak{g}}) \le \operatorname{DIM}(\Pi)$, 
 which is equivalent
 to $n({\mathfrak{g}}_{\mathbb{C}}) \le \operatorname{DIM}(\Pi)$
 in Cases (2) and (3).  
We shall give a brief review
 of several conditions 
 that are equivalent to 
 $n({\mathfrak{g}})> n({\mathfrak{g}}_{\mathbb{C}})$
 in Proposition \ref{prop:OKO}.

We prove the following.  

\begin{theorem}
\label{thm:realmin}
Let 
${\mathbb{O}}
$ be
 a real minimal nilpotent coadjoint orbit in ${\mathfrak{g}}^{\ast}$.  
Then the $K$-action on 
${\mathbb{O}}
$ is coisotropic.  
\end{theorem}

For the proof, 
 we recall some basic facts 
 on real minimal nilpotent orbits.

Let ${\mathfrak{g}}={\mathfrak{k}}+{\mathfrak{p}}$
 be the Cartan decomposition, 
 and $\theta$ the corresponding Cartan involution. 
We take a maximal abelian subspace ${\mathfrak{a}}$
 of ${\mathfrak{p}}$, 
 and fix a positive system $\Sigma^+({\mathfrak{g}}, {\mathfrak{a}})$
 of the restricted root system $\Sigma({\mathfrak{g}}, {\mathfrak{a}})$.  
We denote by $\mu$
 the highest element 
 in $\Sigma^+({\mathfrak{g}}, {\mathfrak{a}})$, 
 and $A_{\mu} \in {\mathfrak{a}}$
 the coroot of $\mu$.  
It is known ({\it{e.g.}}, \cite{O15})
 that any minimal nilpotent coadjoint orbit 
${\mathbb{O}}
$
 is of the form 
 ${\mathbb{O}}
=\operatorname{Ad}(G)X$
 via the identification 
 ${\mathfrak{g}}^{\ast} \simeq {\mathfrak{g}}$
 for some non-zero element 
$
   X \in {\mathfrak{g}}({\mathfrak{a}};\mu)
   :=
\{X \in {\mathfrak{g}}
:
[H,X]=\mu(H) X
\text{ for all $H \in {\mathfrak{a}}$}
\}
$.  
Let $G_X$ be the stabilizer subgroup 
 of $X$ in $G$.  
Then one has the decomposition:

\begin{lemma}
\label{lem:KAN}
$G=K \exp({\mathbb{R}}A_{\mu}) G_X$.  
\end{lemma}

\begin{proof}
We set 
$
   {\mathfrak{a}}^{\perp \mu}
:=
\{H \in {\mathfrak{a}}: \mu(H)=0\}$, 
$
{\mathfrak{n}}=
\bigoplus_{\nu \in \Sigma^+({\mathfrak{g}},{\mathfrak{a}})}
{\mathfrak{g}}({\mathfrak{a}};\nu)$, 
 and 
${\mathfrak{m}}:=Z_{\mathfrak{k}}({\mathfrak{a}})$, 
 the centralizer of ${\mathfrak{a}}$ in ${\mathfrak{k}}$.  
We note that ${\mathfrak{a}}= {\mathbb{R}}A_{\mu} \oplus {\mathfrak{a}}^{\perp \mu}$ is the orthogonal direct sum decomposition
 with respect to the Killing form.

Since $\mu$ is the highest element
 in $\Sigma^+({\mathfrak{g}},{\mathfrak{a}})$, 
 the Lie algebra $Z_{\mathfrak{g}}(X)$
 of $G_X$
 contains ${\mathfrak{a}}^{\perp \mu} \oplus {\mathfrak{n}}$.  
In particular,
 $G_X$ contains the subgroup $\exp({\mathfrak{a}}^{\perp \mu})N$.  
Since $A=\exp({\mathbb{R}}A_{\mu})\exp({\mathfrak{a}}^{\perp \mu})$, 
 the Iwasawa decomposition $G=K A N$ implies
 $G=K \exp ({\mathbb{R}}A_{\mu}) G_X$.  
\end{proof}

\begin{proof}[Proof of Theorem \ref{thm:realmin}]
Retain the above notation and convention.  
In particular, 
 we write as ${\mathbb{O}}=\operatorname{Ad}^{\ast}(G)X$.  
By Lemmas \ref{lem:Hcoiso} and \ref{lem:KAN}, 
 it suffices to verify
\begin{equation}
\label{eqn:coisonu}
({\mathfrak{k}} + Z_{\mathfrak{g}}(X'))^{\perp}
 \subset [X', {\mathfrak{k}}]
\quad
\text{for any $X' \in \operatorname{Ad}(\exp {\mathbb{R}}A_{\mu})X$.  }
\end{equation}
Since $X \in {\mathfrak{g}}({\mathfrak{a}};\mu)$, 
 any
$
   X' \in \operatorname{Ad}(\exp {\mathbb{R}}A_{\beta})X
$
 is of the form $X'=c X$
 for some $c>0$.  
Thus it is enough to show \eqref{eqn:coisonu}
 when $X'=X$.  
Since $Z_{\mathfrak{g}}(X) \supset {\mathfrak{a}}^{\perp\mu} \oplus {\mathfrak{n}}$, 
 one has ${\mathfrak{k}} + Z_{\mathfrak{g}}(X)\supset \theta {\mathfrak{n}} \oplus 
{\mathfrak{a}}^{\perp \mu} \oplus {\mathfrak{m}} \oplus {\mathfrak{n}}$, 
 hence 
 $({\mathfrak{k}} + Z_{\mathfrak{g}}(X))^{\perp \mu} \subset {\mathbb{R}}A_{\mu}$.  
In view that $(A_{\mu}, X, c' \theta X)$ forms
 an ${\mathfrak{sl}}_2({\mathbb{R}})$-triple for some $c'\in {\mathbb{R}}$, 
 one has $A_{\mu} \in [X, {\mathfrak{k}}]$.  
Thus \eqref{eqn:coisonu} is verified for $X'=X$.  
Hence the $K$-action
 on ${\mathbb{O}}
$
 is coisotropic by Lemma \ref{lem:1.7}.  
\end{proof}

\subsection{Complex minimal nilpotent orbit}
\label{subsec:cpxmin}
In this section 
 we give a proof of Theorem \ref{thm:22030921}.

Suppose that $G_{\mathbb{C}}$ is a connected complex simple Lie group.  
We take a Cartan subalgebra ${\mathfrak{h}}_{\mathbb{C}}$
 of the Lie algebra ${\mathfrak{g}}_{\mathbb{C}}$ of $G_{\mathbb{C}}$, 
 choose a positive system
 $\Delta^+({\mathfrak{g}}_{\mathbb{C}}, {\mathfrak{h}}_{\mathbb{C}})$, 
 and set 
$
   {\mathfrak{n}}_{\mathbb{C}}^+:=
   \bigoplus_{\alpha \in \Delta^+({\mathfrak{g}}_{\mathbb{C}},{\mathfrak{h}}_{\mathbb{C}})}
  {\mathfrak{g}}_{\mathbb{C}}({\mathfrak{h}}_{\mathbb{C}};\alpha)$, 
$
   {\mathfrak{n}}_{\mathbb{C}}^-:=
   \bigoplus_{\alpha \in \Delta^+({\mathfrak{g}}_{\mathbb{C}},{\mathfrak{h}}_{\mathbb{C}})}
  {\mathfrak{g}}_{\mathbb{C}}({\mathfrak{h}}_{\mathbb{C}};-\alpha)$.  
Let $\beta$ be the highest root 
in $\Delta^+({\mathfrak{g}}_{\mathbb{C}}, {\mathfrak{h}}_{\mathbb{C}})$, 
 and $H_{\beta} \in {\mathfrak{h}}_{\mathbb{C}}$
 the coroot of $\beta$.  
Then one has the direct sum decomposition
 ${\mathfrak{h}}_{\mathbb{C}}={\mathbb{C}}H_{\beta} \oplus {\mathfrak{h}}_{\mathbb{C}}^{\perp \beta}$
 where ${\mathfrak{h}}_{\mathbb{C}}^{\perp \beta}:=
\{H \in {\mathfrak{h}}_{\mathbb{C}}
:
\beta(H)=0\}$.  
The minimal nilpotent coadjoint orbit
 ${\mathbb{O}}_{\operatorname{min}, {\mathbb{C}}}$
 is of the form 
$
   {\mathbb{O}}_{\operatorname{min}, {\mathbb{C}}}
   =
   \operatorname{Ad}(G_{\mathbb{C}})X
\simeq
G_{\mathbb{C}}/(G_{\mathbb{C}})_X
$
 for any non-zero $X \in {\mathfrak{g}}({\mathfrak{h}}_{\mathbb{C}};\beta)$
 via the identification ${\mathfrak{g}}_{\mathbb{C}}^{\ast} \simeq {\mathfrak{g}}_{\mathbb{C}}$.  
One can also write as 
 ${\mathbb{O}}_{\operatorname{min}, {\mathbb{C}}} = \operatorname{Ad}(G_{\mathbb{C}})Y \simeq G_{\mathbb{C}}/(G_{\mathbb{C}})_Y$
 for any non-zero $Y \in {\mathfrak{g}}({\mathfrak{h}}_{\mathbb{C}};-\beta)$.

By an elementary representation theory of ${\mathfrak{s l}}_2$, 
 one sees ({\it{e.g.,}} \cite{C93})
 that the Lie algebras $Z_{\mathfrak{g}_{\mathbb{C}}}(X)$
 and  $Z_{\mathfrak{g}_{\mathbb{C}}}(Y)$
 of the isotropy subgroups 
$(G_{\mathbb{C}})_{X}$ and $(G_{\mathbb{C}})_{Y}$
 are given respectively by
\begin{align}
\label{eqn:Zglmd}
Z_{\mathfrak{g}_{\mathbb{C}}}(X)
=&
\bigoplus_
{\substack{\alpha \in \Delta^+({\mathfrak{g}}_{\mathbb{C}}, {\mathfrak{h}}_{\mathbb{C}}) \\ \alpha \perp \beta}}
{\mathfrak{g}}_{\mathbb{C}}({\mathfrak{h}}_{\mathbb{C}};-\alpha)
 \oplus {\mathfrak{h}}_{\mathbb{C}}^{\perp \beta} 
\oplus {\mathfrak{n}}_{\mathbb{C}}^+, 
\\
\notag
Z_{\mathfrak{g}_{\mathbb{C}}}(Y)
=&
{\mathfrak{n}}_{\mathbb{C}}^-
\oplus
{\mathfrak{h}}_{\mathbb{C}}^{\perp\beta}
\oplus
\bigoplus_{\substack{\alpha \in \Delta^+({\mathfrak{g}}_{\mathbb{C}}, {\mathfrak{h}}_{\mathbb{C}}) \\ \alpha \perp \beta}}
{\mathfrak{g}}_{\mathbb{C}}({\mathfrak{h}}_{\mathbb{C}};\alpha).  
\end{align}

\begin{proof}[Proof of Theorem \ref{thm:22030921} (1)]

We set $S:=\exp {\mathbb{C}}(H_{\beta}, -H_{\beta}) \cdot (X,Y)$
 in ${\mathbb{O}}_{\operatorname{min}, {\mathbb{C}}} \times {\mathbb{O}}_{\operatorname{min}, {\mathbb{C}}}$.  
We claim 
 that $\operatorname{diag}(G_{\mathbb{C}}) S$ is open dense
 in ${\mathbb{O}}_{\operatorname{min}, {\mathbb{C}}} \times {\mathbb{O}}_{\operatorname{min}, {\mathbb{C}}}$.  
To see this, 
 we observe 
 that $(G_{\mathbb{C}})_X \exp ({\mathbb{C}}H_{\beta})(G_{\mathbb{C}})_Y$
 contains the open Bruhat cell
 $N_{\mathbb{C}}^+ H_{\mathbb{C}} N_{\mathbb{C}}^-
 =N_{\mathbb{C}}^+ \exp({\mathfrak{h}}_{\mathbb{C}}^{\perp\beta})
\exp ({\mathbb{C}}H_{\beta}) N_{\mathbb{C}}^-$
 in $G_{\mathbb{C}}$
 as is seen from \eqref{eqn:Zglmd}, 
 and thus
$
\operatorname{diag}(G_{\mathbb{C}}) \exp {\mathbb{C}}(H_{\beta}, 0) 
((G_{\mathbb{C}})_X \times (G_{\mathbb{C}})_{Y})
$
 is open dense in the direct product group $G_{\mathbb{C}} \times G_{\mathbb{C}}$
 via the identification 
 $\operatorname{diag} (G_{\mathbb{C}}) \backslash (G_{\mathbb{C}} \times G_{\mathbb{C}}) \simeq  G_{\mathbb{C}}$, 
 $(x,y) \mapsto x^{-1} y$.

By Lemma \ref{lem:Hcoiso}, 
 Theorem \ref{thm:22030921} (1) will follow
 if we show
\begin{equation}
\label{eqn:ZZdiag}
   (\operatorname{diag}({\mathfrak{g}}_{\mathbb{C}})
    +
    Z_{{\mathfrak{g}}_{\mathbb{C}}\oplus {\mathfrak{g}}_{\mathbb{C}}}
    (\operatorname{Ad}(a)X,\operatorname{Ad}(a)^{-1}Y))^{\perp B}
   \subset
   [(\operatorname{Ad}(a)X,\operatorname{Ad}(a)^{-1}Y), \operatorname{diag}({\mathfrak{g}}_{\mathbb{C}})]\end{equation}
for any $a \in \exp ({\mathbb{C}} H_{\beta})$.  
Since $\operatorname{Ad}(a)X=cX$
 and $\operatorname{Ad}(a)^{-1}Y=c^{-1}Y$
 for some $c \in {\mathbb{C}}^{\times}$, 
 and since $X$ and $Y$ are arbitrary non-zero elements
 in ${\mathfrak{g}}_{\mathbb{C}}({\mathfrak{h}}_{\mathbb{C}};\beta)$
 and ${\mathfrak{g}}_{\mathbb{C}}({\mathfrak{h}}_{\mathbb{C}};-\beta)$, 
 respectively, 
 it suffices to verify \eqref{eqn:ZZdiag} for $a=e$.  
By \eqref{eqn:Zglmd}, 
 one has 
\[
   (\operatorname{diag}({\mathfrak{g}}_{\mathbb{C}})
    +
    (Z_{{\mathfrak{g}}_{\mathbb{C}}} (X) \oplus Z_{{\mathfrak{g}}_{\mathbb{C}}} (Y)))^{\perp B}
={\mathbb{C}}(H_{\beta},-H_{\beta}).  
\]
Since $[X,Y]=c' H_{\beta}$ 
 for some $c' \in {\mathbb{C}}^{\times}$, 
 one has $[(X,Y), (X+Y, X+Y)]= c' (H_{\beta},-H_{\beta})$, 
 showing $(H_{\beta},-H_{\beta}) \in [(X,Y),\operatorname{diag}({\mathfrak{g}}_{\mathbb{C}})]$.  
Thus Theorem \ref{thm:22030921} (1) is proved.  
\end{proof}

Next, 
 we consider the setting in Theorem \ref{thm:22030921} (2).  
Let $(G_{\mathbb{C}},K_{\mathbb{C}})$ be a symmetric pair
 defined by a holomorphic involutive automorphism $\theta$ of $G_{\mathbb{C}}$.  
Then there is a real form ${\mathfrak{g}}_{\mathbb{R}}$
 of the Lie algebra ${\mathfrak{g}}_{\mathbb{C}}$ of $G_{\mathbb{C}}$
 such that $\theta|_{{\mathfrak{g}}_{\mathbb{R}}}$
 defines the Cartan decomposition
 ${\mathfrak{g}}_{\mathbb{R}}={\mathfrak{k}}_{\mathbb{R}}+{\mathfrak{p}}
_{\mathbb{R}}$
 of the real simple Lie algebra ${\mathfrak{g}}_{\mathbb{R}}$
 with ${\mathfrak{k}}_{\mathbb{R}} \otimes_{\mathbb{R}}{\mathbb{C}}$
 being the Lie algebra ${\mathfrak{k}}_{\mathbb{C}}$
 of $K_{\mathbb{C}}$.  
We denote by $G_{\mathbb{R}}$ 
 the analytic subgroup of $G_{\mathbb{C}}$
 with Lie algebra ${\mathfrak{g}}_{\mathbb{R}}$.

We take a maximal abelian subspace ${\mathfrak{a}}_{\mathbb{R}}$
 in ${\mathfrak{p}}_{\mathbb{R}}$, 
 and apply the results of Section \ref{subsec:realmin}
 by replacing the notation
 ${\mathfrak{g}}$, ${\mathfrak{k}}$, ${\mathfrak{p}}$, ${\mathfrak{a}}$, 
 $\cdots$
 with 
  ${\mathfrak{g}}_{\mathbb{R}}$, ${\mathfrak{k}}_{\mathbb{R}}$, ${\mathfrak{p}}_{\mathbb{R}}$, ${\mathfrak{a}}_{\mathbb{R}}$, etc.

Let ${\mathcal{N}}_{\mathbb{C}}$ be the nilpotent cone 
 in ${\mathfrak{g}}_{\mathbb{C}}$, 
 and ${\mathcal{N}}_{\mathbb{R}, \mathbb{C}}:=
\{X \in {\mathcal{N}}_{\mathbb{C}}:
\operatorname{Ad}(G_{\mathbb{C}})X \cap {\mathfrak{g}}_{\mathbb{R}} \ne \emptyset\}$.  
Then there exists a unique $G_{\mathbb{C}}$-orbit, 
 to be denoted by 
${\mathbb{O}}_{\operatorname{min}, {\mathbb{R}}}^{\mathbb{C}}$, 
 which is minimal
 in $({\mathcal{N}}_{\mathbb{R}, \mathbb{C}} \setminus \{0\})/G_{\mathbb{C}}$
 with respect to the closure relation, 
 and 
 ${\mathbb{O}}_{\operatorname{min}, {\mathbb{R}}}^{\mathbb{C}}=\operatorname{Ad}(G_{\mathbb{C}})X$
 for any non-zero $X \in {\mathfrak{g}}_{\mathbb{R}}({\mathfrak{a}}_{\mathbb{R}};\beta)$
 (\cite{O15}).

We extend ${\mathfrak{a}}_{\mathbb{R}}$
 to a maximally split Cartan subalgebra
 ${\mathfrak{h}}_{\mathbb{R}}={\mathfrak{t}}_{\mathbb{R}}
+{\mathfrak{a}}_{\mathbb{R}}$
 of ${\mathfrak{g}}_{\mathbb{R}}$
 where ${\mathfrak{t}}_{\mathbb{R}}:={\mathfrak{h}}_{\mathbb{R}}\cap 
{\mathfrak{k}}_{\mathbb{R}}
$, 
 write ${\mathfrak{h}}_{\mathbb{C}}={\mathfrak{t}}_{\mathbb{C}} + {\mathfrak{a}}_{\mathbb{C}}$
 for the complexification, 
 and take a positive system 
 $\Delta^+({\mathfrak{g}}_{\mathbb{C}}, {\mathfrak{h}}_{\mathbb{C}})$
 which is compatible
 with $\Sigma^+({\mathfrak{g}}_{\mathbb{R}}, {\mathfrak{a}}_{\mathbb{R}})$.

The proof of Theorem \ref{thm:realmin} shows
 its complexified version as follows.  

\begin{theorem}
\label{thm:22030921b}
The action of $K_{\mathbb{C}}$ on ${\mathbb{O}}_{\operatorname{min}, {\mathbb{R}}}^{\mathbb{C}}$
 is coisotropic.  
\end{theorem}

This confirms Theorem \ref{thm:22030921} (2)
 when 
 ${\mathbb{O}}_{\operatorname{min}, {\mathbb{C}}}
  =
  {\mathbb{O}}_{\operatorname{min}, {\mathbb{R}}}^{\mathbb{C}}$, 
 or equivalently, 
 in Cases (2) and (3) of Lemma \ref{lem:CRmin}.

Let us verify Theorem \ref{thm:22030921} (2)
 in the case 
 ${\mathbb{O}}_{\operatorname{min}, {\mathbb{C}}} \ne 
  {\mathbb{O}}_{\operatorname{min}, {\mathbb{R}}}^{\mathbb{C}}$.

We need the following:

\begin{proposition}
[{\cite[Cor.\ 5.9]{KO15}, \cite[Prop.\ 4.1]{O15}}]
\label{prop:OKO}
Let ${\mathfrak{g}}_{\mathbb{R}}$ be a real form
 of a complex simple Lie algebra
 ${\mathfrak{g}}_{\mathbb{C}}$, 
 and ${\mathfrak{k}}_{\mathbb{C}}$
 the complexified Lie algebra
 of ${\mathfrak{k}}_{\mathbb{R}}$, 
 the Lie algebra ${\mathfrak{k}}_{\mathbb{R}}$
 of a maximal compact subgroup $K_{\mathbb{R}}$
 of the analytic subgroup $G_{\mathbb{R}}$
 in $\operatorname{Int}{\mathfrak{g}}_{\mathbb{C}}$.  
Then the following six conditions
 on ${\mathfrak{g}}_{\mathbb{R}}$ 
 are equivalent:
\begin{enumerate}
\item[{\rm{(i)}}]
${\mathbb{O}}_{\operatorname{min}} \cap {\mathfrak{g}}_{\mathbb{R}}
 = \emptyset$.  
\item[{\rm{(ii)}}]
${\mathbb{O}}_{\operatorname{min}, {\mathbb{C}}} \ne 
  {\mathbb{O}}_{\operatorname{min}, {\mathbb{R}}}^{\mathbb{C}}$.  
\item[{\rm{(iii)}}]
$\theta \beta \ne -\beta$.  
\item[{\rm{(iv)}}]
$n({\mathfrak{g}})>n({\mathfrak{g}}_{\mathbb{C}})$.  
\item[{\rm{(v)}}]
${\mathfrak{g}}_{\mathbb{R}}$ is compact
 or is isomorphic to ${\mathfrak{su}}^{\ast}(2n)$, 
 ${\mathfrak{so}}(n-1,1)$ $(n \ge 5)$, 
 ${\mathfrak{sp}}(m,n)$, 
 ${\mathfrak{f}}_{4(-20)}$, 
 or ${\mathfrak{e}}_{6(-26)}$.  
\item[{\rm{(vi)}}]
${\mathfrak{g}}_{\mathbb{C}}={\mathfrak{k}}_{\mathbb{C}}$
 or the pair $({\mathfrak{g}}_{\mathbb{C}}, {\mathfrak{k}}_{\mathbb{C}})$
 is isomorphic to 
$({\mathfrak{s l}}(2n, {\mathbb{C}}), {\mathfrak{s p}}(n,{\mathbb{C}}))$, 
$({\mathfrak{s o}}(n, {\mathbb{C}}), {\mathfrak{s o}}(n-1,{\mathbb{C}}))$ 
 $(n \ge 5)$, 
$({\mathfrak{s p}}(m+n, {\mathbb{C}}), {\mathfrak{s p}}(m,{\mathbb{C}}) \oplus {\mathfrak{s p}}(n, {\mathbb{C}}))$, 
$({\mathfrak{f}}_4^{\mathbb{C}}, {\mathfrak{s o}}(9,{\mathbb{C}}))$, 
 or 
$({\mathfrak{e}}_{6}^{\mathbb{C}}, {\mathfrak{f}}_{4}^{\mathbb{C}})$.  
\end{enumerate}

\end{proposition}

\begin{remark}
The equivalence (i) $\iff$ (v) was stated 
 in \cite[Prop.\ 4.1]{B98}
 without proof, 
 and Okuda \cite{O15} supplied a complete proof.  
\end{remark}

\begin{lemma}
\label{lem:22031312}
Suppose $X$ is a highest root vector, 
 namely, 
 $0 \ne X \in {\mathfrak{g}}_{\mathbb{C}}({\mathfrak{h}}_{\mathbb{C}}
;\beta)$.  
If $\theta \beta \ne - \beta$, 
 then $H_{\beta} \in {\mathfrak{k}}_{\mathbb{C}}+Z_{\mathfrak{g}_{\mathbb{C}}}(X)$.  
\end{lemma}

\begin{proof}
Since $\theta \beta \ne - \beta$, 
 one has $\beta|_{\mathfrak{t}_{\mathbb{C}}} \not \equiv 0$, 
 namely, 
 ${\mathfrak{t}}_{\mathbb{C}} \not \subset {\mathfrak{h}}_{\mathbb{C}}^{\perp \beta}$.  
Since ${\mathfrak{h}}_{\mathbb{C}}^{\perp \beta}$ is of codimension one
 in ${\mathfrak{h}}_{\mathbb{C}}$, 
 we get ${\mathfrak{t}}_{\mathbb{C}} + {\mathfrak{h}}_{\mathbb{C}}^{\perp \beta}={\mathfrak{h}}_{\mathbb{C}}$.  
Thus $H_{\beta} \in {\mathfrak{h}}_{\mathbb{C}} \subset {\mathfrak{k}}_{\mathbb{C}} + Z_{\mathfrak{g}_{\mathbb{C}}}(X)$.  
\end{proof}

\begin{proposition}
\label{prop:22031313}
If one of (and therefore any of) the equivalent conditions
 in Proposition \ref{prop:OKO} holds,
 then $K_{\mathbb{C}}$ has a Zariski open orbit
 in ${\mathbb{O}}_{\operatorname{min}, {\mathbb{C}}}$.  
In particular, 
 the $K_{\mathbb{C}}$-action 
 on ${\mathbb{O}}_{\operatorname{min}, {\mathbb{C}}}$ is coisotropic.  
\end{proposition}

\begin{proof}
Since ${\mathbb{O}}_{\operatorname{min}, {\mathbb{C}}}=\operatorname{Ad}(G_{\mathbb{C}})X$
 for a non-zero $X \in {\mathfrak{g}}_{\mathbb{C}}({\mathfrak{h}}_{\mathbb{C}};\beta)$, 
 the proposition is clear.  
\end{proof}

\begin{proof}[Proof of Theorem \ref{thm:22030921} (2)]
The Case (1) in Lemma \ref{lem:CRmin} is proved in Proposition \ref{prop:22031313}, 
 and the Cases (2) and (3) are proved in Theorem \ref{thm:22030921b}.  
\end{proof}

\subsection{Proof of Theorems in Section \ref{sec:Intro}}
As we saw at the end of Section \ref{sec:2}, 
Theorems \ref{thm:Joseph} and \ref{thm:tensormin}
 are derived from the geometric result, 
 namely, from Theorem \ref{thm:22030921}, 
 and thus the proof of these theorems
 has been completed.

In the same manner, 
 one can deduce readily from Theorem \ref{thm:22030921b}
 the following bounded multiplicity property
 which is not covered by Theorem \ref{thm:Joseph}
 for the five cases
 in Proposition \ref{prop:OKO}
 where $n({\mathfrak{g}})>n({\mathfrak{g}}_{\mathbb{C}})$.

\begin{theorem}
\label{thm:Rminbdd}
Suppose that the Gelfand--Kirillov dimension 
 of $\Pi \in \operatorname{Irr}(G)$
 is $n({\mathfrak{g}})$.  
If $(G,G')$ is a symmetric pair
 such that ${\mathfrak{g}}_{\mathbb{C}}'$ is conjugate
 to ${\mathfrak{k}}_{\mathbb{C}}$
 by $\operatorname{Int}{\mathfrak{g}}_{\mathbb{C}}$, 
 then $m(\Pi|_{G'})<\infty$.  
\end{theorem}

\begin{proof}
We write $G_{\mathbb{C}}'$ and $K_{\mathbb{C}}$
 for the analytic subgroups
 of $G_{\mathbb{C}}=\operatorname{Int}{\mathfrak{g}}_{\mathbb{C}}$
 with Lie algebras ${\mathfrak{g}}_{\mathbb{C}}'$
 and ${\mathfrak{k}}_{\mathbb{C}}$, 
 respectively.  
Then the $K_{\mathbb{C}}$-action
 on ${\mathbb{O}}_{\operatorname{min}, {\mathbb{R}}}^{\mathbb{C}}$
 is coisotropic by Theorem \ref{thm:22030921b}, 
 and so is the $G_{\mathbb{C}}'$-action
 on ${\mathbb{O}}_{\operatorname{min}, {\mathbb{R}}}^{\mathbb{C}}$
 because $G_{\mathbb{C}}'$ and $K_{\mathbb{C}}$
 are conjugate by an element of $G_{\mathbb{C}}$.  
Hence the theorem follows from Fact \ref{fact:Ki}.  
\end{proof}

Finally, 
 we give a proof of Theorem \ref{thm:irr}.  
\begin{proof}
[Proof of Theorem \ref{thm:irr}]
Let ${\mathcal{J}}$
 be the Joseph ideal.  
Let 
$(U({\mathfrak{g}}_{\mathbb{C}})/{\mathcal{J}})^{{\mathfrak{g}}_{\mathbb{C}}'}$
be the algebra of ${\mathfrak{g}}_{\mathbb{C}}'$-invariant elements
 in $U({\mathfrak{g}}_{\mathbb{C}})/{\mathcal{J}}$
 via the adjoint action.  
Then one has 
\[
(U({\mathfrak{g}}_{\mathbb{C}})/{\mathcal{J}})^{{\mathfrak{g}}_{\mathbb{C}}'}
={\mathbb{C}}
\]
if one of (therefore, all of)
 the equivalent conditions
 in Proposition \ref{prop:OKO}
 is satisfied, 
 see \cite[Lem.\ 3.4]{Ta}.  
In particular, 
the center $Z({\mathfrak{g}}_{\mathbb{C}}')$
 of the enveloping algebra $U({\mathfrak{g}}_{\mathbb{C}}')$
 of the subalgebra ${\mathfrak{g}}_{\mathbb{C}}'$
 acts as scalars
 on the minimal representation $\Pi$
 because the action factors
 through
 the following composition of homomorphisms:
\[
Z({\mathfrak{g}}_{\mathbb{C}}') \to 
U({\mathfrak{g}}_{\mathbb{C}})/{\mathcal{J}}
\to
\operatorname{End}_{\mathbb{C}}(\Pi).  
\]
Since any minimal representation is unitarizable
 by the classification \cite{Ta}, 
 and since there are at most finitely many elements 
 in $\operatorname{Irr}(G')$ 
having a fixed 
$Z({\mathfrak{g}}_{\mathbb{C}}')$-infinitesimal character, 
 the restriction $\Pi|_{G'}$ splits
 into a direct sum of at most finitely many 
 irreducible representations of $G'$, 
 with multiplicity being finite 
 by Theorem \ref{thm:minbdd}.  
Thus the proof of Theorem \ref{thm:irr} is completed.  
\end{proof}

\end{document}